\documentclass[10pt,reqno]{article}

\usepackage[b5paper,top=1.2in,left=0.9in]{geometry}
\usepackage{verbatim}
\usepackage{amssymb}
\usepackage{amsmath}
\usepackage{graphicx}
\usepackage{appendix}
\usepackage{color}
\usepackage{amsthm}

\renewcommand{\d}{\mathrm{d}}
\newcommand{\D}{\mathrm{D}}

\newcommand{\e}{\mathrm{e}}

\newtheorem{Thm}{Theorem}[section]
\newtheorem{Lem}[Thm]{Lemma}
\newtheorem{Prop}[Thm]{Proposition}
\newtheorem{Cor}[Thm]{Corollary}
\newtheorem{Rem}[Thm]{Remark}
\newtheorem{Def}[Thm]{Definition}

\newtheorem{Fact}[Thm]{Fact}
\newtheorem{Nota}[Thm]{Notation}

\newtheorem*{MThm}{Main Theorem}

\def\R{\mathbb{R}}
\def\Q{\mathbb{Q}}
\def\N{\mathbb{N}}
\def\C{\mathbb{C}}
\def\Z{\mathbb{Z}}
\def\T{\mathbb{T}}
\def\P{\mathbb{P}}

\def\to{\longrightarrow}

\def\cU{\mathcal{U}}

\def\a{\alpha}
\def\b{\beta}
\def\e{\epsilon}

\def\D{\Delta}

\def\d{\delta}

\def\l{\lambda}

\def\fb{\mathfrak{b}}
\def\sl{\mathfrak{sl}}
\def\g{\mathfrak{g}}

\def\fh{\mathfrak{h}}
\def\fq{\mathfrak{q}}

\def\ox{\otimes}
\def\o+{\oplus}
\def\bo+{\bigoplus}
\def\x{\times}
\def\p[#1,#2]{\phi_{#1,#2}}
\def\til[#1]{\widetilde{#1}}
\def\what[#1]{\widehat{#1}}

\def\bu{\textbf{u}}
\def\bv{\textbf{v}}

\def\be{\textbf{e}}
\def\bE{\textbf{E}}
\def\bf{\textbf{f}}
\def\bF{\textbf{F}}

\def\bi{\textbf{i}}

\def\bK{\textbf{K}}

\def\bp{\textbf{p}}

\def\bU{\textbf{U}}

\def\b0{\textbf{0}}

\def\z[#1]{z_{#1}}

\def\oo{\infty}
\def\=>{\Longrightarrow}
\def\inj{\hookrightarrow}

\def\<{\langle}
\def\>{\rangle}

\def\^{\wedge}
\def\+{\dagger}

\def\inv{^{-1}}
\def\dis{\displaystyle}
\def\over[#1]{\overline{#1}}
\def\vec[#1]{\overrightarrow{#1}}
\def\mat[#1, #2]{\left[\begin{array}{ccccc}#1\end{array}\left|\begin{array}{c}#2\end{array}\right.\right]}
\def\xto[#1]{\xrightarrow{#1}}
\def\dd[#1,#2]{\frac{d#1}{d#2}}
\def\del[#1,#2]{\frac{\partial #1}{\partial #2}}
\def\Facts[#1]{\begin{Fact}\mbox{}\begin{itemize}#1\end{itemize}\end{Fact}}
\def\Notation[#1]{\begin{Nota}\mbox{}\begin{itemize}#1\end{itemize}\end{Nota}}
\def\Eqn[#1]{\begin{eqnarray*}#1\end{eqnarray*}}
\def\tab{\;\;\;\;\;\;}

\newcommand{\veca}[2][cccccccccccccccccccccccccccccccccccccccccc]{\left(\begin{array}{#1}#2 \\ \end{array} \right)}

\newcommand{\Eq}[1]{\begin{align}#1\end{align}}
\newcommand{\case}[2][cccccccccccccccccccccccccccccccccccccccccc]{\left\{\begin{array}{#1}#2 \\ \end{array}\right.}

\begin{document}
\title{Positive representations of split real quantum groups and future perspectives}
\author{Igor B. Frenkel\footnote{Email: igor.frenkel@yale.edu}, \tab Ivan C.H. Ip\footnote{ Email: ivan.ip@yale.edu} \\\\
Yale University, \\Department of Mathematics,\\10 Hillhouse Ave,\\New Haven, \\CT 06520,  U.S.A.}

\date{\today}

\numberwithin{equation}{section}

\maketitle

\begin{abstract}
We construct a special principal series representation for the modular double $\bU_{\fq\til[\fq]}(\g_\R)$ of type $A_r$ representing the generators by positive essentially self-adjoint operators satisfying the transcendental relations that also relate $q$ and $\til[q]$. We use the cluster variables parametrization of the positive unipotent matrices to derive the formulas in the classical case. Then we quantize them after applying the Mellin transform. Our construction is inspired by the previous results for $\g_\R=\sl(2,\R)$ and is expected to have a generalization to other simply-laced types. We conjecture that our positive representations are closed under the tensor product and we discuss the future perspectives of the new representation theory following the parallel with the established developments of the finite-dimensional representation theory of quantum groups.
\end{abstract}

\newpage
\tableofcontents
\section{Introduction\label{sec:intro}}
In their foundational papers Drinfeld \cite{D} and Jimbo \cite{J} have defined for any finte dimensional complex simple Lie algebra $\g$ (and more generally for any Kac-Moody algebra) a remarkable Hopf algebra $\cU_q(\g)$ known as quantum group. As the notation indicates the quantum group $\cU_q(\g)$ is a deformation of the universal enveloping algebra $\cU(\g)$ for a nonzero complex parameter $q$. They were also able to deform the irreducible finite-dimensional representations of $\cU(\g)$ to corresponding representations of $\cU_q(\g)$, which stay irreducible when $q$ is not a root of unity. These representations as in the classical case are parametrized by the cone of the positive weights $P^+\subset \fh_\R^*$, where $\fh_\R$ is the real form of the Cartan subalgebra $\fh\subset\g$. These representations have a Hermitian form compatible with the quantum counterpart of the canonical Hermitian conjugation on $\g$ and $\cU(\g)$. Let $\g_c\subset \g$ be a compact real form fixed by the classical Hermitian conjugation, and let $\cU_q(\g_c)$ denote the quantum group $\cU_q(\g)$ equipped with the corresponding Hermitian structure that is well defined *-Hopf algebra for the real nonzero parameter $q$ \cite{Tw}.

It is natural to consider other real forms of $\g$, most notably the split real form $\g_\R\subset\g$, and address the question about the $q$-deformation of its irreducible unitary representations. Since the works of Drinfeld and Jimbo the $q$-deformation of various infinite-dimensional irreducible representations were found \cite{Kl,MMNNSU,Sch}. However, the general problem of the $q$-deformations of all unitary irreducible representations of $\g$ seems to be too difficult and we have to be content to consider special classes of representations. For the split real form $\g_\R$ there is one distinguished series of irreducible unitary representations associated to the minimal parabolic or Borel subalgebra $\fb_\R$ parametrized by the $\R_+$-span $P_\R^+\subset\fh_R^*$ of the discrete cone $P^+$. This series, usually called the minimal principal series, also constitutes the decomposition of $L^2(G_\R/K)$, where $G_\R$ is the Lie group corresponding to $\g_\R$ and $K$ is its maximal compact subgroup \cite{H1}, and it also can be viewed as the most continuous series in the decomposition of $L^2(G_\R)$ \cite{H2}. In this paper we present a construction of what we view as the most canonical $q$-deformation of this distinguished series of unitary representations in the case $\g_\R=\sl(n,\R)$ and we suggest a generalization to the case of an arbitrary simply-laced split real form $\g_\R$.

In the case of the split real form the Hermitian conjugation on $\cU_q(\g)$ is well defined for $q$ on the unit circle \cite{Tw}, and let $\cU_q(\g_\R)$ denote again the quantum group with this extra structure. Here we consider
\Eq{q=e^{\pi \bi b^2},\tab \til[q]=e^{\pi \bi b^{-2}},}
where $\bi=\sqrt{-1}$, $b^2\in\R\setminus\Q$ such that $q$ and $\til[q]$ are not roots of unity. The starting point of our construction was the work of Teschner et al \cite{BT, PT1, PT2} who studied extensively a very special "$q$-deformation" of the principal series of representations of the quantum group $\cU_q(\sl(2,\R))$ in the space $L^2(\R)$. Although the formula is a $q$-deformation of the classical action, the parameter $1+\l=\frac{1}{2}+\bi\a$ that gives a unitary representation for $SL(2,\R)$ is also perturbed so that this special series admit the parameter $1+\l=\frac{1}{2}+\frac{1}{2b^2}+\bi\a$. The parameter of the representation appears in formulas with the factor $b$ so that its real part gives $\frac{Q}{2}=\frac{1}{2}(b+b\inv)$, which has no classical limit when
\Eq{b\to0,\tab q\to1,} 
nor the limit corresponding to $b\to\oo$, $\til[q]\to1$. However, what we have gained is that the action aquires a duality between $b\leftrightarrow b\inv$, and the operators become positive self-adjoint, and one can discuss its functional calculus. Details on the functional analysis of unbounded operators, the self-adjointness, as well as the important Lemma \ref{b2lem} can be found for example in \cite{Ip,Ru}. In particular, these representations are naturally extended to the modular double \Eq{\cU_{q\til[q]}(\sl(2,\R))=\cU_q(\sl(2,\R))\ox\cU_{\til[q]}(\sl(2,\R))} of the quantum group first introduced by Faddeev \cite{Fa1,Fa2}. The modular double has two sets of mutually commuting generators $\{E,F,K^{\pm 1}\}$ and $\{\til[E],\til[F],\til[K]^{\pm 1}\}$ satisfying the quantum group relations
\begin{eqnarray}
\label{Uq1} KE&=&q^2EK,\\
\label{Uq2}KF&=&q^{-2}FK,\\
\label{Uq3}EF-FE&=&\frac{K-K\inv}{q-q\inv},
\end{eqnarray}
and similarly for the second set with tildes. To formulate some special additional properties of these representations it is convenient to introduce the rescaled generators
\Eq{e=2\sin(\pi b^2) E,\tab f=2\sin(\pi b^2)F,}
and similarly for the tilde set. Then the representations of the modular double $\cU_{q\til[q]}(\sl(2,\R))$ in $L^2(\R)$ possess the following properties:
\begin{itemize}
\item[(i)] the generators $e,f,K^{\pm 1}$ and $\til[e],\til[f],\til[K]^{\pm 1}$ are represented by positive essentially self-adjoint operators,
\item[(ii)] the generators satisfy the transcendental relations
\Eq{\label{trans}e^{\frac{1}{b^2}}=\til[e],\tab f^{\frac{1}{b^2}}=\til[f],\tab K^{\frac{1}{b^2}}=\til[K].}
\end{itemize}

Our generalization of the Teschner et al construction to the modular double $\cU_{q\til[q]}(\g_\R)$ of the quantum group associated to Lie algebra $\g_\R$ of type $A_r$ of rank $r$ and dimension $r+2N$ is completely analogous to the $\sl(2,\R)$ case. In particular the functional analysis can be reduced to the $\cU_{q\til[q]}(\sl(2,\R))$ case. However in higher rank, there appear new algebraic features. Specifically, the extra relations 
\Eq{K_iE_j=q^{a_{ij}}E_jK_i,\tab K_iF_j=q^{-a_{ij}}F_jK_i,}
where $(a_{ij})$ is the Cartan matrix, and the quantum Serre relations do not allow adjacent variables
$\{E_i, F_i, K_i\}$ and $\{\til[E_j],\til[F_j],\til[K_j]\}$ to commute whenever $|i-j|=1$. To remedy this, we have to introduce a slightly modified version of the quantum group so that the tilde variables commute with the original variables. We define
\Eq{\fq:=q^2=e^{2\pi \bi b^2},\tab \til[\fq]:=\til[q]^2=e^{2\pi \bi b^{-2}},}
and the $\fq$-commutator
\Eq{[A,B]_\fq=AB-\fq\inv BA.} Also let $\T_{\fq\til[\fq]}^{n(n-1)/2}$ be the quantum tori generated by positive self-adjoint operators $\bu_{ij},\bv_{ij},\til[\bu_{ij}], \til[\bv_{ij}]$ for $1\leq i<j\leq n$ such that
\Eq{\bu_{ij}\bv_{ij}=\fq \bv_{ij}\bu_{ij}, \tab \til[\bu_{ij}]\til[\bv_{ij}]=\til[\fq]\til[\bv_{ij}]\til[\bu_{ij}].}

\begin{MThm} Let $\{\bE_i,\bF_i,\bK_i^{\pm1}\}_{i=1}^r$ and $\{\til[\bE]_i,\til[\bF]_i,\til[\bK]_i^{\pm1}\}_{i=1}^r$ be two sets of mutually commuting generators of the modified modular double $\bU_{\fq\til[\fq]}(\g_\R)$ where $\g_\R$ is of type $A_r$, satisfying the relations
\Eq{\bK_i\bE_j=\fq^{a_{ij}}\bE_j\bK_i,\tab \bK_i\bF_j=\fq^{-a_{ij}}\bF_j\bK_i,}
the modified relations
\Eq{[\bE_i,\bF_i]_\fq=\frac{1-\bK_i}{1-\fq},}
as well as the modified quantum Serre relations
\Eq{[\bE_i,[\bE_{i+1},\bE_i]_\fq]=0=[\bE_{i+1},[\bE_{i+1},\bE_i]_\fq],}
\Eq{[\bF_i,[\bF_{i},\bF_{i+1}]_\fq]=0=[\bF_{i+1},[\bF_i,\bF_{i+1}]_\fq],}
and similarly for the second set with tildes. Then there exist a family of irreducible representations of $\bU_{\fq\til[\fq]}(\g_\R)$ parametrized by $\l\in P_\R^+$ on the space $L^2(\R^N)$ with the additional properties (i) and (ii) for $\{\bE_i,\bF_i,\bK_i^{\pm1}\}_{i=1}^r$. Moreever, there is an embedding
\Eq{\bU_{\fq\til[\fq]}(\sl(n,\R))\inj \T_{\fq\til[\fq]}^{n(n-1)/2}.}
\end{MThm}

Though we prove the theorem for the type $A_r$, we expect that it is true for any simply-laced type. In our proof of the theorem we are able to present explicit expressions for the generators and verify directly all the relations and properties. Our verification of the commutation relations, both in the classical and quantum case, is based on a new pictorial method, which we believe presents an independent interest. We also provide a derivation of our formulas in the classical case using a parametrization of the positive unipotent matrices by the cluster variables associated to the canonical orientation of the $A_r$ quiver
\Eq{\label{Arorientation}\circ_1\to\circ_2\to\cdots\circ_{r-1}\to\circ_r}
or its opposite. Then using the positivity properties we rewrite our formulas by applying the Mellin transform. Finally using the rules of the $q$-deformation inspired by the $\sl(2,\R)$ case studied by Teschner et al we obtain the desired representations of the modular double $\cU_{q\til[q]}(\sl(n,\R))$ and its modification $\bU_{\fq\til[\fq]}(\sl(n,\R))$.

To prove our theorem for other types of simply-laced Lie algebras one can also try to choose the cluster variables associated to a canonical orientation of a quiver with a source or a sink at the branching point, however one should expect substantially more complicated formulas than for $A_r$ type. A more conceptual approach should come from the theory of total positivity developed by Lusztig \cite{Lu} and the use of cluster variables associated to different orientations of a quiver as well as relations between them.  Since the $q$-deformed cluster variables for $GL_q^+(n,\R)$ still commute up to powers of $q$, as constructed in \cite{Ip2}, one can, in principle, derive our formulas for the quantum generators directly. Note that for different orientations of a quiver different generators of the quantum group admit especially simple expressions. In particular, for the canonical orientation (resp. its opposite) of the $A_r$ quiver as shown above, the generators $E_r$ and $F_1$ (resp. $E_1$ and $F_r$) contain only one shifting operators.

Though in our paper we construct representations of the modular double $\bU_{\fq\til[\fq]}(\g_\R)$ by a certain deformation of the representations of the classical Lie algebra the actual relation between the quantum and classical cases is rather mysterious. Although there is a formal classical limit in the non-perturbed case when we consider a fixed generic complex parameter $\l$, there is no straightforward classical limit $b\to 0, q\to 1$ when we pass to the positive setting since $\til[q]$ "blows up" as we discussed earlier. It is an interesting problem how to "extract" the classical theory from its quantum counterpart.

The class of representations of the modular double considered in this paper also plays an important role for the deformation of the space of functions on the split real group $G_\R$. Since we always impose the requirement of positive definiteness of quantum generators it is more natural in our setting to consider the deformation of the space of functions on the positive semigroup $G_\R^+\subset G_\R$, which we denote by $F_{q\til[q]}(G_\R^+)$. The construction of this space using quantum cluster variables is proposed in \cite{BZ}. In the case when $G_\R=SL(2,\R)$ it was conjectured by Teschner \cite{PT1} and proved by the second author \cite{Ip} that the space $F_{q\til[q]}(G_\R^+)$ under the regular representation of $\cU_{q\til[q]}(\g_\R)$, and a suitable choice of $L^2$ structure, is decomposed into a direct integral of irreducible representations precisely given by our theorem. It is natural to conjecture that it is also true for the higher rank case, where the space $F_{q\til[q]}(GL_q^+(n,\R))$ equipped with a suitable $L^2$ norm is constructed by the second author explicitly in \cite{Ip2}. Again it would be interesting to compare the classical and quantum cases by characterizing the restriction of the most continuous component of $L^2(G_\R)$ to $G_\R^+$.

Since the positivity properties of generators in our representations of the modular double $\cU_{q\til[q]}(\g_\R)$, as well as  $\bU_{\fq\til[\fq]}(\g_\R)$, play a crucial role we call them \emph{positive principal series representations} or just \emph{positive representations}. Note that there are other ways to deform the principal series of representations even associated with the same minimal parabolic subalgebra $\fb_\R^+\subset\g_\R$. For example a class of representations of $\cU_{q\til[q]}(\sl(n,\R))$ has been constructed in \cite{GKL} but the generators do not seem to be represented by positive self-adjoint operators. Another example of a principal series representation for $\cU_q(\sl(n,\R))$ is constructed in \cite{ANO} using $q$-difference operators. However the variables the authors are using come from the standard coordinates of $U^+$ of the Gauss decomposition, which do not admit the construction of action by positive operators, and hence do not extend to a representation of the modular double.

The plan of the paper is as follows. In Section \ref{sec:urep} we construct the minimal principal series representation for $\cU(\sl(n))$ using the parametrization of totally positive matrix by cluster variables. Then we perform the Mellin transform and obtain a realization of the action using shifting operators. To motivate the calculations of the quantum case, we introduce the commutation relation diagrams for the actions and prove directly all the Lie algebra relations including the Serre relations. In Section \ref{sec:uqrep} we quantize the formulas obtained above, and generalizing the rank 1 case, we construct the positive principal series representations such that the action of $\cU_q(\sl(n,\R))$ is realized by positive essentially self-adjoint operators. We show that our construction is naturally extended to the modular double with the desired transcendental relations \eqref{trans}. In Section \ref{sec:main} we introduce the modified quantum generators to obtain the commutativity between the modular double variables, and present our main theorem. Finally in Section \ref{sec:future} we discussed various future perspectives of the current program.

\textbf{Acknowledgements.} The results of this paper were announced at the workshops in Banff and Aarhus in August and October 2011, respectively. We are grateful to the organizers for creative atmosphere that led to interesting discussions of our results. The first author was supported by the NFS grant DMS-100163.

\section{Principal series representations of $\cU(\sl(n))$\label{sec:urep}}
\subsection{Total positivity and cluster variables}
Total positivity for general reductive group is considered by Lusztig \cite{Lu}. In the case for $G=GL(n,\R)$, a matrix is totally positive if all its entries and the determinant of the minors are positive. Furthremore, the positive monoid admits the Gauss decomposition $GL^+(n,\R)=U_{>0}^-T_{>0}U_{>0}^+$, where $U_{>0}^\pm$ are totally positive upper/lower triangular matrices (considered only for the upper/lower triangular minors), and $T_{>0}$ are diagonal matrices with positive entries.

In \cite{BFZ}, another parametrization using cluster variables are studied. These are given by the "initial minors" that are determinants of the square submatrices which start from either the top row or the leftmost column. Restricted to the upper triangular unipotent $U_{>0}^+$, the cluster variables are $x_{i,j}$, $1\leq i<j\leq n$, where $x_{i,j}$ is the determinant of the initial minor
\Eq{x_{i,j}=\det \veca{z_{1,j-i+1}&...&z_{1,j}\\\vdots&\ddots&\vdots\\z_{i,j-i+1}&\cdots&z_{i,j}}.}
This parametrization correspond to the canonical decomposition of the maximal Weyl group element $w_0$ as
\Eq{w_0=s_{n-1}s_{n-2}...s_2s_1s_{n-1}s_{n-2}...s_2s_{n-1}s_{n-2}...s_3...s_{n-1},}
where $s_k=(k, k+1)$ are the standard transpositions, so that
\Eq{U_{>0}^+=\left\{ \prod_{l=1}^{n-1}\prod_{k=1}^{n-l} s_{n-k}(a_{n-k,n-k-l+1})\bigg|a_{ij}>0\tab\mbox{ for } 1\leq j\leq i\leq n-1\right\},}
with 
\Eq{s_{i}(t)=I_n+t E_{i,i+1},} and $E_{i,j}$ is the standard matrix with 1 at the entry $(i,j)$ and 0 otherwise. Then there is a 1-1 correspondence between $a_{ij}$ and $x_{ij}$ given by
\begin{Prop} We have
\Eq{a_{i,j}=\frac{x_{j,i+1}x_{j-1,i-1}}{x_{j,i}x_{j-1,i}},}
\Eq{x_{i,i+j}=\prod_{m=1}^j \prod_{n=1}^i a_{m+n-1,n}.}
Here we denote by $x_{i,i}=x_{i,0}=x_{0,j}=1$.
\end{Prop}
Furthermore, by calculating the Jacobian of the change of variables from the standard coordinates $z_{ij}$ to the cluster variables $x_{ij}$, we have
\begin{Prop}\label{Haar} The Haar measure on $U_{>0}^+$ induced by $\prod_{1\leq i<j\leq n}dz_{ij}$ on $U^+$ is given by
\Eq{\prod_{1\leq i<j\leq n}\frac{dx_{ij}}{x_{ij}}\prod_{i=1}^{n-1}dx_{in}.}
\end{Prop}

\subsection{Infinitesimal action\label{sec:infinite}}
The minimal principal series representation for $\cU(\sl(n,\R))$ can be realized on the totally positive matrices as the infinitesimal action of $g\in SL^+(n,\R)$ acting on $\C[U_{>0}^+]$ by
\Eq{g\cdot f(g_+) =\chi_\l(g_+g) f([g_+g]_+).}
Here we write the Gauss decomposition of $g$ as
\Eq{g=g_-g_0g_+,}
so that $[g]_+=g_+$ is the projection of $g$ onto $U_{>0}^+$, and $\chi_\l(g)$ is the character function defined by
\Eq{\chi_\l(g)=\prod_{i=1}^n u_i^{2\l_i},}
where $\l=(\l_i)\in \C^n$ and $u_i$ is the entry of the diagonal part $g_0\in T_{>0}$.

For a general matrix, the projection onto $U^+$ of the Gauss decomposition is given by:
\begin{Lem}
The entry $z_{ij}$ of $[g]_+$ is given by
\Eq{\frac{\det N_i^j}{\det N_i},}
where $N_i$ is the $i\x i$ determinant of the main diagonal minor of $g$, and $N_i^j$ is the minor $N_i$ with the last column replaced by the $j$-th column $\<g_{1j},...,g_{ij}\>^T$.
\end{Lem}

Now we can find the action of $\exp(t X)\in SL(n,\R)$ and hence $X\in \sl(n,\R)$ by infinitesimal method.

First we consider $X=E_i$. The elementry matrix
\Eq{\exp(tE_i)=\veca{I_{i-1}&0&0 & 0\\0&1&t&0\\0&0&1&0\\0&0&0&I_{n-i-1}}=I+te_{i,i+1}}
only modifies the $i+1$-th column. Therefore we can immediately read off its action:
\begin{Prop}
For $1\leq j<k\leq n$ one has
\Eq{\exp(tE_i)\cdot x_{jk}=\left\{\begin{array}{cc}x_{jk}&\mbox{if $k-j\neq i$,}\\x_{jk}+N_{i;j}t &\mbox{if $k-j=i$,}\end{array}\right.}
where $N_{i;j}$ is the determinant of the original $j\x(j+1) $ block matrix from $z_{1,i}$ to $z_{j,i+j}$ with the second column removed:
\Eq{N_{i;j}=\det \veca{z_{1,i}&z_{1,i+2}&\cdots&z_{1,i+j}\\\vdots&\vdots&\ddots&\vdots\\z_{j,i}&z_{j,i+2}&\cdots&z_{j,i+j}}.}

In particular, $N_{1;1}=1$ and $N_{i;1}=x_{1i}$ for $i>1$.
\end{Prop}

Next we consider $X=F_i$. The elementry matrix
\Eq{\exp(tF_i)=\veca{I_{i-1}&0&0 & 0\\0&1&0&0\\0&t&1&0\\0&0&0&I_{n-i-1}}=I+te_{i+1,i}}
only modifies the $i$-th column.

Since $F_i$ is lower triangular, the action will induce lower triangular term where only a single entry is off. Therefore applying the projection formulas as above, the entries can be easily determined:

\begin{Lem} The projection $g_+$ of $g$ under the action of $F_i$ is given by
\Eq{\left([\exp(tF_i)\cdot g]_+\right)_{jk}:=\case{z_{jk}&\mbox{if $j<i$ and $k\neq i$,}\\z_{ji}+z_{j,i+1}t&\mbox{if $j<i$ and $k= i$,}\\\\\dis\frac{z_{jk}}{1+z_{i,i+1}t}&\mbox{if $j=i$,}\\\\z_{jk}+t \det\veca{z_{i,i+1}&z_{i,k}\\1&z_{i+1,k}}&\mbox{if $j=i+1$,}\\z_{jk}&\mbox{if $j>i+1$.}}}
\end{Lem}
\begin{proof}
We note that the denominator for the projection formula is 1 unless $j=i$, which induces the factor $1+z_{i,i+1}t$ in the diagonal part. Therefore the formula follows from a simple determinant calculation.
\end{proof}
The diagonal factor $1+z_{i,i+1}t$ can be combinand with the character function, and we obtain
\begin{Prop} For $1\leq j<k\leq n$ one has
\Eq{\exp(tF_i)\cdot x_{jk}=\case{x_{jk} &\mbox{if $j<i$ and $k\neq i$,}\\x_{jk}+N^{i;j}t &\mbox{if $j<i$ and $k=i$,}\\x_{ik}(1+z_{i,i+1}t)^{2\l_i-1}&\mbox{if $j=i$,}\\x_{jk}&\mbox{if $j>i$,}\\}}
where $N^{i;j}$ is the determinant of the original $j\x(j+1) $ block matrix from $z_{1,i-j+1}$ to $z_{j,i+1}$ with the second to last column removed:
\Eq{N^{i;j}=\det \veca{z_{1,i-j+1}&\cdots&z_{1,i-1}&z_{1,i+1}\\\vdots&\ddots&\vdots&\vdots\\z_{j,i-j+1}&\cdots&z_{j,i-1}&z_{j,i+1}}.}
In particular $N^{i;1}=x_{1,i+1}$.
\end{Prop}

Finally the action of
\Eq{\exp(tH_i)=\veca{I_{i-1}&0&0 & 0\\0&e^t&0&0\\0&0&e^{-t}&0\\0&0&0&I_{n-i-1}}}
can also be easily found:

\begin{Prop} The action of $\exp(tH_i)$ is given by:
\Eq{\exp(tH_i)\cdot x_{jk}=\case{e^{2\l_i t}e^t x_{jk} &\mbox{if $k=i$,}\\e^{2\l_i t}e^{-t}x_{jk}&\mbox{if $j=i$ or $k-j=i$ but not both,}\\ e^{2\l_i t}e^{-2t}x_{jk}&\mbox{if $j=i,k=2i$,}\\ e^{2\l_i t}x_{jk}&\mbox{otherwise}.}}
\end{Prop}

Combining the above propositions, we obtain the action of the Lie algebra generators:

\begin{Thm} The action of $E_i,F_i$ and $H_i$ are given by
\begin{eqnarray}
E_i\cdot f&=& \sum_{j=1}^{n-i} N_{i;j}f_{j,i+j},\\
F_i\cdot f&=& -\sum_{k=i+1}^{n} x_{ik}z_{i,i+1}f_{ik}+\sum_{j=1}^{i-1} N^{i;j} f_{ji} +2z_{i,i+1}\l_1,\\
H_i\cdot f&=& \sum_{j=1}^{i-1} x_{ji}f_{ji}- \sum_{k=i+1}^{n} x_{ik}f_{ik}-\sum_{j=1}^{n-i}x_{j,i+j}f_{j,i+j}+2\l_i.
\end{eqnarray}
\end{Thm}

Following the techniques from \cite{BFZ}, the auxillary terms $N_{i;j},N^{i;j}$ and $z_{i,i+1}$ can actually be expressed in terms of $x_{ij}$:
\begin{Prop}
We have the following expressions:
\begin{eqnarray}
a_{i,j}&=&\frac{x_{j-1,i-1}x_{j,i+1}}{x_{j-1,i}x_{j,i}},\\
z_{i,i+1}&=&\sum_{j=1}^i a_{i,j}=\sum_{j=1}^i \frac{x_{j,i+1}x_{j-1,i-1}}{x_{j,i}x_{j-1,i}},\\
N_{i;j}&=&x_{j,i+j}\sum_{k=1}^j \frac{x_{k-1,i+k}x_{k,i+k-1}}{x_{k-1,i+k-1}x_{k,i+k}},\\
N^{i;j}&=&x_{j,i}\sum_{k=1}^j a_{i,k}=x_{j,i}\sum_{k=1}^j \frac{x_{k-1,i-1}x_{k,i+1}}{x_{k-1,i}x_{k,i}},
\end{eqnarray}
where $x_{0,j}=x_{k,k}=1$.
\end{Prop}
\subsection{Mellin transformed action\label{sec:Mellin}}
In the classical theory of $SL(2,\R)$, the Mellin transform is used to study the matrix coefficients, see for example \cite{Vi}. This transformation is valid because we are working with positive variables, and it enables us to express differential operators in terms of shifting operators.

Using the formal Mellin transform, we can look at the action of $U(\sl(n,\R))$ as shifting operators with scalar weights:
\begin{eqnarray*}
\del[,x]\int f(\mu)x^{u}du&=&\int (u) f(u)x^{u-1}d\mu=\int (u+1) f(u+1)x^{u}du,\\
\end{eqnarray*}
\Eq{\mbox{or }\tab\del[,x]: f(u)\mapsto (u+1)f(u+1).}
Similarly:
\begin{eqnarray}
x:f(u)&\mapsto& f(u-1),\\
x\del[,x]: f(u)&\mapsto&uf(u),\\
&etc.\nonumber&
\end{eqnarray}

Therefore according to the expression of $z_{i,i+1},N_{i;j},N^{i;j}$, each monomial term in the action of $E_i, F_i$ involves at most 4 shifting operators. We write explicitly the action under this transform as follows.

\begin{Thm}\label{lieaction} The action of $E_i, F_i$ and $H_i$ are given by:
\begin{eqnarray}
E_i \cdot f(\bu)&=&\sum_{k=1}^{n-i}\left(1+\sum_{j=k}^{n-i}u_{j,i+j}\right)f(u_{k-1,i+k-1}+1,u_{k-1,i+k}-1,u_{k,i+k-1}-1,u_{k,i+k}+1),\nonumber\\\\
F_i\cdot f(\bu)&=&\sum_{k=1}^{i}\left(1+\sum_{j=k}^{i}u_{ji}-\sum_{j=i+1}^{n}u_{ij}+2\l_i\right) f(u_{k-1,i-1}-1,u_{k-1,i}+1,u_{k,i}+1,u_{k,i+1}-1),\nonumber\\\\
H_i\cdot f(\bu)&=&\left(\sum_{j=1}^{i-1}u_{ji}-\sum_{j=i+1}^{n}u_{ij}-\sum_{j=1}^{n-i}u_{j,i+j}+2\l_i\right)f(\bu),
\end{eqnarray}
where $u_{0,i}=u_{k,k}=0$ and the shifting operators at these indices are non-existent.
\end{Thm}

Formally this formula is nothing but the shifting operator induced by polynomials in $x_{ij}$. However, when $u_{ij}$ has the appropriate real and imaginery part, Mellin transform can be carried out in the $L^2(\R)$ sense, and these operators will become positive self-adjoint operators. These observations will be studied in Section \ref{sec:slnpositive}.

\subsection{The Lie algebra relations}
The Lie algebra axioms are automatically satisfied for the action in Theorem \ref{lieaction}, since they have arised from the standard infinitesimal action for $SL(n,\R)$. However, the relations will not be guaranteed anymore when we try to quantize the above action. Hence we first directly verify these relations in this classical setting, and we will observe that the quantum case is completely analogous.

Let us introduce the following notations for the action in Theorem \ref{lieaction}:
\begin{eqnarray}
E_i\cdot f(\bu)&=&\sum_{k=1}^{n-i}E_i^k(\bu) f(\bu+\be_{E_i}^k),\\
F_i\cdot f(\bu)&=&\sum_{k=1}^{i}F_i^k(\bu) f(\bu+\be_{F_i}^k),\\
H_i\cdot f(\bu)&=&H_i(\bu)f(\bu).
\end{eqnarray}

In order to calculate the commutation relation, it is useful to introduce the commutation relation (CR) diagrams for $E_i^k, F_i^k$ and $H_i$ (see Figure 1 and 2).

\begin{figure}
\centering
\includegraphics[width=50mm]{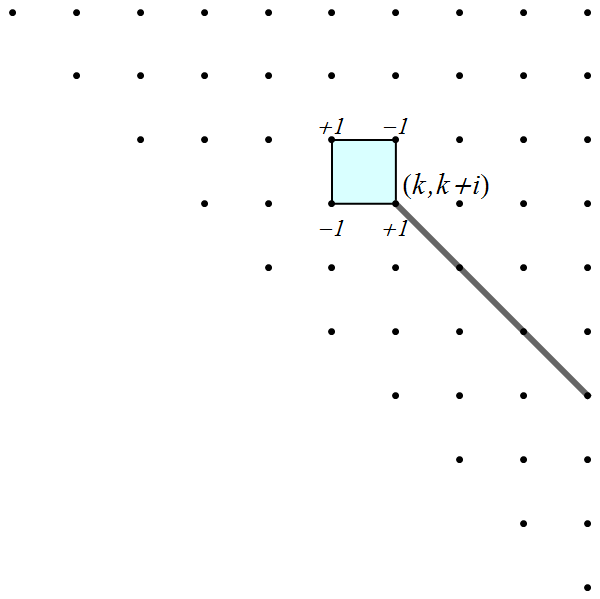}\tab\tab\tab
\includegraphics[width=50mm]{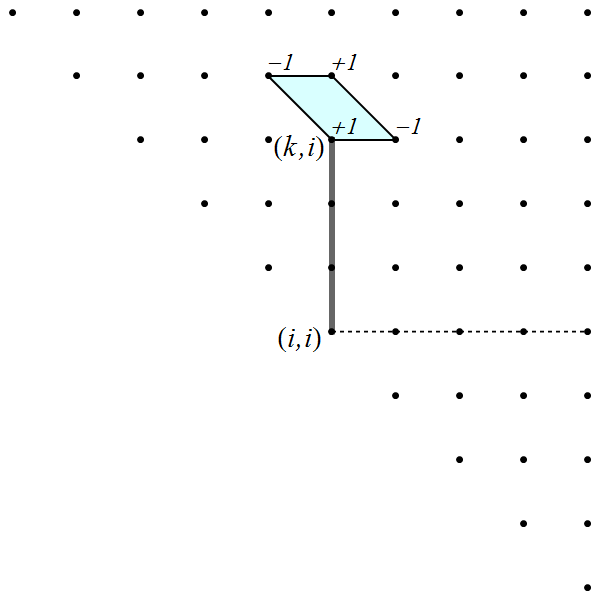}
\caption{The CR diagrams for $E_i^k$ and $F_i^k$}
\includegraphics[width=50mm]{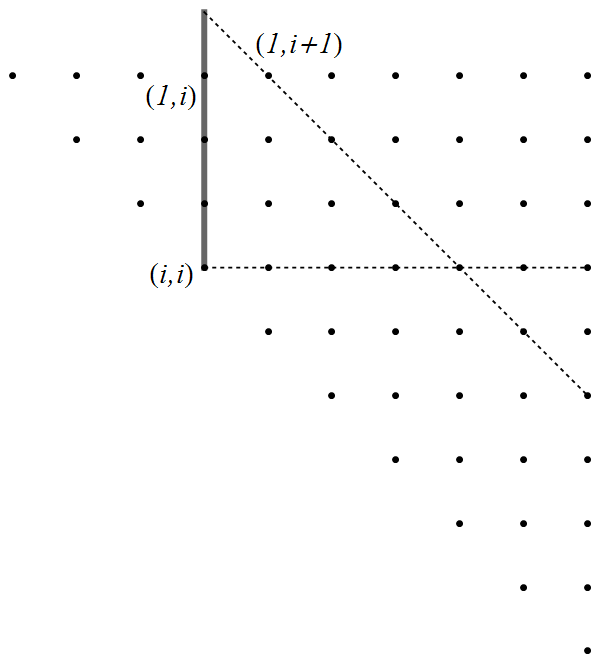}
\caption{The CR diagram for $H_i$}
\end{figure}
Here the quadrilateral encodes the shifting of the operator, and always lies within the grids, and also the "0-th row" when $k=1$. So, for example, the operator $E_i^k$ involves shifting in $u_{k,k+i},u_{k-1,k+i-1}$ by $+1$, and $u_{k-1,k+i},u_{k,k+i-1}$ by $-1$. The straight lines encode the multiplication of weights, where the solid lines indicate posiitve combinations, while the dashed lines indicate negative combinations. So for example the dashed line of $F_i^k$ means $-\sum_{j=i}^n u_{ij}$. Note that the weight is 0 (hence the term is actually not there) where the solid and dashed lines meet.

Now we can compute the commutation relation by the following equation:
\begin{Lem} \label{rel1} Let $P_i$, $i=1,2$, be the operators
\Eq{P_i\cdot f(\bu)=P_i(\bu)f(\bu+\be_i),} where $P_i(\bu)$ are linear functions. Then
\Eq{\label{cr}[P_1, P_2]\cdot f(\bu)=(P_1(\bu)P'_2(\be_1)-P_2(\bu)P'_1(\be_2))f(\bu+\be_1+\be_2),}
where $P'_i(\be_j)=P_i(\be_j)-P_i(\b0)$, i.e. it ignores the constant term in the expression of $P_i(\bu)$.
\end{Lem}
\begin{proof} Follows from linearity of $P_i(\bu)$.
\end{proof}

Note that in our case, the expression $P'_i(\be_j)$ can be found by composing the quadrilateral from the CR diagram of $P_j$ to the solid-dashed lines from the CR diagram of $P_i$, and summing up all the weights. So for example. $F_i^k(\be_{E_{i'}}^{k'})$ is given by composing the square from $E_{i'}^{k'}$ onto the lines for $F_i^k$, and it will only pick up a nonzero sum of weights when $k=k'$ and $k+i=i'$.

\begin{Lem} \label{value} We have the following values:
\begin{eqnarray}
{E'}_i^k(\be_{F_{i'}}^{k'})&=&\case{+1& (k',i')=(k,i+k)\\-1& (k',i')=(k,i+k-1)\\0&\mbox{otherwise,}}\\
{F'}_{i'}^{k'}(\be_{E_i}^k)&=&\case{+1& (k,k+i)=(k',i')\\ -1& (k,k+i)=(k',i'+1)\\0&\mbox{otherwise,}}\\
H'_i(\be_{E_{i'}}^k)&=&\case{+1&|i-i'|=1\\-2 & i=i'\\ 0& \mbox{otherwise,}}\\
H'_i(\be_{F_{i'}}^k)&=&\case{-1&|i-i'|=1 \\2 & i=i'\\ 0& \mbox{otherwise.}}
\end{eqnarray}
\end{Lem}
\begin{proof} Follows from a direct inspection of the CR diagrams.
\end{proof}

\begin{Prop} We have
\begin{eqnarray}
[H_i,E_j]&=&a_{ij}E_j,\\
{[H_i,F_j]}&=&-a_{ij}F_j,
\end{eqnarray}
where $a_{ij}=\case{2&i=j\\-1&|i-j|=1\\0&\mbox{otherwise}}$ is the Cartan Matrix.
\end{Prop}
\begin{proof}
Since $H_i$ does not have a shift, by Lemma \ref{rel1} we have
$$[H_i, E_{i'}^k] = -E_{i'}^k(\bu) H'_i(\be_{E_{i'}}^k)f(\bu+\be_{E_{i'}}) = -H'_i(\be_{E_{i'}}^k)E_{i'}^k,$$
for all $k$. Hence summing up $k$ gives the required relation by Lemma \ref{value}. Similarly for $F_i$.
\end{proof}

\begin{Prop} We have
\Eq{[E_i,F_{i'}]=\d_{ii'}H_i.}
\end{Prop}
\begin{proof} By Lemma \ref{value}, we observe that $E_i^k(\be_{F_{i'}}^{k'})=F_{i'}^{k'}(\be_{E_i}^k)$ identically for all $i,i',k,k'$. Hence by Lemma \ref{rel1}, the relation reduces to:
$$[E_i^k,F_{i'}^{k'}]\cdot f(\bu)=(E_i^k(\bu)-F_{i'}^{k'}(\bu)) E_i^k(\be_{F_{i'}}^{k'})f(\bu+\be_{E_i}^k+\be_{F_{i'}}^{k'}).$$
When $i>i'$, it is clear that $E_i^k(\be_{F_{i'}}^{k'})=0$ because $k+i>i$, so the nonzero cases can never be satisfied.
When $i<i'$, only $k=k'$ gives nonzero value of $E_i^k(\be_{F_{i'}}^{k'})=\d_{i',k+i}-\d_{i',k+i-1}$. Hence we just need to consider the two terms for $k=i'-i$ and $k=i'-i+1$. Next we note that
$$\be=\be_{E_{i}}^{i'-i}+\be_{F_{i'}}^{i'-i}=\be_{E_{i}}^{i'-i+1}+\be_{F_{i'}}^{i'-i+1},$$
$$f(\bu+\be)= f(u_{i'-i,i'-1}-1,u_{i'-i,i'}+2,u_{i'-i,i'+1}-1),$$
therefore all we need to take care of is the factor. We have
\begin{eqnarray*}
E_i^{i'-i}(\bu)-F_{i'}^{i'-i}(\bu)
&=&\sum_{j=i'-i}^{n-i}u_{j,i+j}-\sum_{j'=i'-i}^{i'}u_{j'i'}+\sum_{j'=i'+1}^{n}u_{i'j'}+2\l_{i'}\\
&=&\sum_{j=i'-i+1}^{n-i}u_{j,i+j}-\sum_{j'=i'-i+1}^{i'}u_{j'i'}+\sum_{j'=i'+1}^{n}u_{i'j'}+2\l_{i'}\\
&=&E_i^{i'-i+1}(\bu)-F_{i'}^{i'-i+1}(\bu),
\end{eqnarray*}
hence we conclude that the factor equals:
\begin{eqnarray*}
&&\sum_{k=i'-i}^{i'-i+1}(E_i^k(\bu)-F_{i'}^{k}(\bu))E_i^k(\be_{F_{i'}}^{k})\\
&=&\sum_{k=i'-i}^{i'-i+1}(E_i^k(\bu)-F_{i'}^{k}(\bu))(\d_{i',k+i}-\d_{i',k+i-1})\\
&=&(E_i^{i'-i}(\bu)-F_{i'}^{i'-i}(\bu))-(E_i^{i'-i+1}(\bu)-F_{i'}^{i'-i+1}(\bu))\\
&=&0.
\end{eqnarray*}
Finally, when $i=i'$, $E_i^k(\be_{F_{i}}^{k'})=0$ unless $k=k'$ and $i=i+k-1$, i.e. $k=k'=1$ which gives $E_i^1(\be_{F_{i}}^{1})=-1$. Furthermore, $\be_{E_i}^1+\be_{F_i}^1=\b0$. Hence we just need to  calculate
\begin{eqnarray*}
[E_{i}^1,F_{i'}^1]&=&(E_i^1(\bu)-F_i^1(\bu))(-1)\\
&=&\sum_{j=1}^{i-1}u_{ji}-\sum_{j=i+1}^{n}u_{ij}-\sum_{j=1}^{n-i}u_{j,i+j}+2\l_i\\
&=&H_i.
\end{eqnarray*}
\end{proof}

Next we verify the Serre relations. Since they involve relations in $E$ only, or $F$ only, we calculate the $P'(\be)$ functions with respect to the shifts within $E_i$ and $F_i$, respectively.

\begin{Lem}\label{value2} We have the following values:

We have for $E_i$:
\begin{eqnarray}
E_i^k(\be_{E_i'}^{k'})=\left\{\begin{array}{cc}0&k>k'\\1&k=k'\\2&k<k',\end{array}\right.\\
E_i^k(\be_{E_{i+1}}^{k'})=\left\{\begin{array}{cc}0&k>k'\\-1&k\leq k',\end{array}\right.\\
E_{i+1}^k(\be_{E_i}^{k'})=\left\{\begin{array}{cc}0&k\geq k'\\-1&k< k'.\end{array}\right.
\end{eqnarray}
We have similarly for $F_i$:
\begin{eqnarray}
F_i^k(\be_{F_i}^{k'})=\left\{\begin{array}{cc}0&k>k'\\1&k=k'\\2&k<k',\end{array}\right.\\
F_i^k(\be_{F_{i+1}}^{k'})=\left\{\begin{array}{cc}0&k\geq k'\\-1&k< k',\end{array}\right.\\
F_{i+1}^k(\be_{F_i}^{k'})=\left\{\begin{array}{cc}0&k> k'\\-1&k\leq k'.\end{array}\right.
\end{eqnarray}
\end{Lem}
\begin{proof} This again follows from a direct inspection of the diagrams: how square superimposes on the diagonal line, and how parallelgram superimposes on the solid-dashed lines.
\end{proof}

\begin{Cor} When $|i-j|\geq 2$, $[E_i,E_j]=[F_i,F_j]=0$.
\end{Cor}
\begin{proof} The square or parallelgram never touches the lines, so all the $P_i(\be_j)$'s are zero in the commutation relation \eqref{cr}.
\end{proof}

\begin{Lem}
We have the Serre relations
\Eq{\left\{\begin{array}{c}E_iE_iE_{i+1}-2E_iE_{i+1}E_i+E_{i+1}E_iE_i=0,\\F_iF_iF_{i+1}-2F_iF_{i+1}F_i+F_{i+1}F_iF_i=0,\end{array}\right.}

\Eq{\left\{\begin{array}{c}E_{i+1}E_{i+1}E_i-2E_{i+1}E_iE_{i+1}+E_iE_{i+1}E_{i+1}=0,\\F_{i+1}F_{i+1}F_i-2F_{i+1}F_iF_{i+1}+F_iF_{i+1}F_{i+1}=0.\end{array}\right.}
\end{Lem}
\begin{proof} First we observe that for both $P=E$ or $F$,
$$P_i^k(\be_{P_i}^{k'})+P_i^{k'}(\be_{P_i}^k)=2,$$
$$P_i^k(\be_{P_i}^{k''})+P_{i+1}^{k''}(\be_{P_i}^k)=-1.$$
Hence we let $a=P_i^{k'}(\be_{P_i}^k),b=P_{i+1}^{k''}(\be_{P_i}^k),c=P_{i+1}^{k''}(\be_{P_i}^{k'})$. Note that $a$ can only take values $0,1,2$, while $b,c$ can only take $0,-1$. Then the Serre relations amount to the vanishing of $B(a,b,c)$ given by
\begin{eqnarray}
B(a,b,c)&=&\sum_{sym\{k,k'\}}\bigg(P_i^k(\bu)(P_{i}^{k'}(\bu)+a)(P_{i+1}^{k''}(\bu)+b+c)\nonumber\\
&&-2 P_i^{k}(\bu)(P_{i+1}^{k''}(\bu)+b)(P_i^{k'}(\bu)+a-c-1)\nonumber\\
&&+P_{i+1}^{k''}(\bu)(P_i^k(\bu)-b-1)(P_i^{k'}(\bu)+a-c-1)\bigg).\label{Babc}
\end{eqnarray}
Here the sum is a symmetrized sum where in the second set $k,k'$ are interchanged, $b,c$ are interchanged, and $a\to 2-a$.

After expanding and simplifying, it becomes
\Eq{B(a,b,c)=(P_i^k(\bu)+P_i^{k'}(\bu)+P_{i+1}^{k''}(\bu)) (b(2-a+c)+(a+b)c).}
Hence the Serre relations amount to the equation
\Eq{\label{Serreabc}b(2-a+c)+(a+b)c=0.}

It is easy to see that $B(a,0,0)=B(a,-1,-1)=0$ for all $a$, and $B(0,0,-1)=B(2,-1,0)=0$. Therefore it remains to check that the parameter must fall into these cases.

Indeed, if $a=0$, then $k'>k$, hence $(0,-1,0)$ cannot happen because otherwise we have $k'<k''<k$. Similarly $(2,0,-1)$ cannot happen. If $a=1$, then $k'=k$ and we must have $k''<k=k'$ or $k''>k=k'$ giving $(1,0,0)$ or $(1,-1,-1)$.

This proves the first Serre relations for both $E_i$ and $F_i$. The second Serre relations are exactly the same with $i$ replaced by $i+1$, and keeping $a,b,c$ in the same order.
\end{proof}
\section{Principal series representations of $\cU_q(\sl(n))$\label{sec:uqrep}}
\subsection{Quantization\label{sec:quan}}
Assured by the CR diagram s, we can now quantize the action and repeat the same proofs for all the quantum group relations. The procedure is just to change the weight into its quantum number as follows:
\begin{Thm} We have the action of $\cU_q(\sl(n))$ given by:
\begin{eqnarray}
E_i f(\bu)&=&\sum_{k=1}^{n-i}\left[1+\sum_{j=k}^{n-i}u_{j,i+j}\right]_qf(u_{k-1,i+k-1}+1,u_{k-1,i+k}-1,u_{k,i+k-1}-1,u_{k,i+k}+1),\nonumber\\\\
F_i f(\bu)&=&\sum_{k=1}^{i}\left[1+\sum_{j=k}^{i}u_{ji}-\sum_{j=i+1}^{n}u_{ij}+2\l_i\right]_q f(u_{k-1,i-1}-1,u_{k-1,i}+1,u_{k,i}+1,u_{k,i+1}-1),\nonumber\\\\
K_i f(\bu)&=&q^{\left(\sum_{j=1}^{i-1}u_{ji}-\sum_{j=i+1}^{n}u_{ij}-\sum_{j=1}^{n-i}u_{j,i+j}+2\l_i\right)}f(\bu),
\end{eqnarray}
where $[n]_q=\frac{q^n-q^{-n}}{q-q^{-1}}$. They satisfy all the quantum commutation relation \eqref{Krel}-\eqref{qserre2}.
\end{Thm}
\begin{proof} We need to check all the quantum relations. Let us denote by $E_i^k(\bu),F_i^k(\bu)$ as before, so that for $P=E$ and $F$,
$$P_i\cdot f(\bu)=\sum_k [P_i^k(\bu)]_qf(\bu+\be_i^k).$$

First note that
\Eq{\label{qnumber101}[n]_q=n\tab\mbox{for $n=0,1,-1$}.}

The relations with $K_i$:
\begin{eqnarray}\label{Krel}
K_iE_j&=&q^{a_{ij}}E_jK_i,\\
K_iF_j&=&q^{-a_{ij}}F_jK_i
\end{eqnarray}
follow from the classical calculations because $K_i=q^{H_i}$, where $H_i$ is just the classical action.

The relations
\Eq{[E_i,F_j]=\d_{ij}\frac{K_i-K_i\inv}{q-q\inv}}
follow from the fact that when $E_i^k(\be_{F_{i'}}^{k'})=F_{i'}^{k'}(\be_i^k)$ as in the classical case, the commutation relation factor becomes
$$[E_i^k(\bu)]_q[F_{i'}^{k'}(\bu+\be_{E_i}^k)]_q-[F_{i'}^{k'}(\bu)]_q[E_i^k(\bu+\be_{F_{i'}}^{k'})]_q=[E_i^k(\bu)-F_{i'}^{k'}(\bu)]_q[E_i^k(\be_{F_{i'}}^{k'})]_q.$$

Hence all the classical calculations still hold, including the case $i=i'$, where we get
$$[E_i^1(\bu)-F_i^1(\bu)]_q=[H_i]_q=\frac{K_i-K_i\inv}{q-q\inv}.$$

Finally the quantum Serre relations:
\Eq{\label{qserre1}\left\{\begin{array}{c}E_iE_iE_{i+1}-[2]_qE_iE_{i+1}E_i+E_{i+1}E_iE_i=0,\\F_iF_iF_{i+1}-[2]_qF_iF_{i+1}F_i+F_{i+1}F_iF_i=0,\end{array}\right.}
\Eq{\label{qserre2}\left\{\begin{array}{c}E_{i+1}E_{i+1}E_i-[2]_qE_{i+1}E_iE_{i+1}+E_iE_{i+1}E_{i+1}=0,\\F_{i+1}F_{i+1}F_i-[2]_qF_{i+1}F_iF_{i+1}+F_iF_{i+1}F_{i+1}=0,\end{array}\right.}
where $[2]_q=q+q\inv$, are equivalent to the vanishing of the commutation factor (with $P=E$ or $F$)

\begin{eqnarray*}
B_q(a,b,c)&=&\sum_{sym\{k,k'\}}\bigg([P_i^k(\bu)]_q[P_{i}^{k'}(\bu)+a]_q[P_{i+1}^{k''}(\bu)+b+c]_q\\
&&-[2]_q [P_i^{k}(\bu)]_q[P_{i+1}^{k''}(\bu)+b]_q[P_i^{k'}(\bu)+a-c-1]_q\\
&&+[P_{i+1}^{k''}(\bu)]_q[P_i^k(\bu)-b-1]_q[P_i^{k'}(\bu)+a-c-1]_q\bigg)
\end{eqnarray*}
with the same $a,b,c$ and the symmetrized sum as before (cf. \eqref{Babc}).

It turns out that this can also be simplified by expanding, and we obtain
\Eq{B_q(a,b,c)=[P_i^k(\bu)+P_i^{k'}(\bu)+P_{i+1}^{k''}(\bu)]_q ([b]_q[2-a+c]_q+[a+b]_q[c]_q),}
completely analogous to the classical case.
Hence the quantum Serre relation amounts to the equation
\Eq{[b]_q[2-a+c]_q+[a+b]_q[c]_q=0,}
which is equivalent to the classical equation \eqref{Serreabc} using \eqref{qnumber101}.
\end{proof}

\subsection{Positive representations of $\cU_q(\sl(2,\R))$ \label{sec:sl2positive}}
Let us motivate our construction of the positive principal series representations by considering first the transition from $\cU(\sl(2))$ to $\cU_q(\sl(2,\R))$. For $\cU(\sl(2))$, the action of the generators is given by
\begin{eqnarray*}
E\cdot f(u)&=&(u+1)f(u+1),\\
F\cdot f(u)&=&(1-u+2\l)f(u-1),\\
H\cdot f(u)&=&(-2u+2\l)f(u).
\end{eqnarray*}
Now note that for $SL_2^+(\R)$, the Haar measure on $U_{>0}^+$ is given by $du$, hence when we apply the Mellin transform, we actually want $Re(u)=-\frac{1}{2}$ in order for the $L^2$ structure be preserved. Hence if we make the translation $u\to -\bi u+\l$ with $\l=-\frac{1}{2}+\bi \a$, $\a\in\R$, we obtain:
\begin{eqnarray*}
E\cdot f(u)&=&(\frac{1}{2}+\bi \a-\bi u)f(u+\bi ),\\
F\cdot f(u)&=&(\frac{1}{2}+\bi \a+\bi u)f(u-\bi ),\\
H\cdot f(u)&=&2\bi uf(u),
\end{eqnarray*}
which can then be seen to be anti self-adjoint, unbounded operators, so that their exponentials are unitary operators.

In the work \cite{PT2}, the (Fourier transformed) action for $\cU_q(\sl(2,\R))$, where $q=e^{\pi \bi b^2}$ for $0<b<1$, is constructed by scaling $u$ by $b$ and replacing $\frac{1}{2}+\bi \a$ by $\frac{Q}{2b}+\bi \frac{\a}{b}$ in the quantized formula, where $Q=b+b\inv$, so that the action becomes:
\begin{eqnarray}
E&=&[\frac{Q}{2b}+\frac{\bi }{b}(\a-u)]_qe^{-2\pi bp},\\
F&=&[\frac{Q}{2b}+\frac{\bi }{b}(\a+u)]_qe^{2\pi bp},\\
K&=&e^{2\pi bu},
\end{eqnarray}
where $p=\frac{1}{2\pi \bi }\dd[,u]$ so that $e^{\pm2\pi bp}f(u)=f(u\mp \bi b)$.

This representation, called in \cite{PT2} the self dual principal series, has two remarkable properties.
First the action given is positive essentially self-adjoint. Due to the factor $\frac{Q}{2b}=\frac{1}{2}+\frac{1}{2b^2}$, the expression for $E$ and $F$ is actually:
\begin{eqnarray}
E&=&\left(\frac{\bi }{q-q\inv}\right) (e^{\pi b(\a-u-2p)}+e^{-\pi b(\a-u+2p)}),\\
F&=&\left(\frac{\bi }{q-q\inv}\right) (e^{\pi b(\a+u+2p)}+e^{-\pi b(\a+u-2p)}),
\end{eqnarray}
which is positive essentially self-adjoint, as shown in \cite{Ip}. Note that the factor
\Eq{\left(\frac{\bi }{q-q\inv}\right)=(2\sin(\pi b^2))\inv}
is positive for $0<b<1$.

Secondly, it is self dual under the change $b\longleftrightarrow b\inv$ in the following sense. This change gives the action $(\til[E],\til[F],\til[K])$ of its modular double counterpart $\cU_{\til[q]}(\sl(2,\R))$ where $\til[q]=e^{\pi \bi b^{-2}}$, which commute with $(E,F,K)$ weakly (i.e. the spectrum doesn't commute). Furthermore, if we let
\Eq{e=2\sin(\pi b^2)E,\tab f= 2\sin(\pi b^2)F,}
\Eq{\til[e]=2\sin(\pi b^{-2})\til[E],\tab \til[f]= 2\sin(\pi b^{-2})\til[F],}
then the following transcendental relations are valid:
\Eq{e^{\frac{1}{b^2}}=\til[e],\tab f^{\frac{1}{b^2}}=\til[f],\tab K^{\frac{1}{b^2}}=\til[K].}
The proof is based on the following Lemma that is also repeatedly used in our construction

\begin{Lem}\label{b2lem}\cite{BT, Ip} If $u,v$ are essentially self-adjoint and $uv=q^2vu$, then
$u+v$ is essentiall self-adjoint, and
\Eq{(u+v)^{1/b^2}=u^{1/b^2}+v^{1/b^2}.}
\end{Lem}
\subsection{Positive representations of $\cU_q(\sl(n,\R))$ \label{sec:slnpositive}}
In order to construct a positive representation, we need to shift our (pure imaginary) variables $u_{ij}$ with appropriate real part such that the Mellin transform preserves the Haar measure, and moreover each expression in the quantum weight has the factor
$\frac{1}{2}+\bi \a_k$ where $\l_k=-\frac{1}{2}+\bi \a_k$ with $\a_k\in\R$.

\begin{Thm} There is a unique shift in $u_{i,j}\to -\bi u_{i,j}+c_{i,j}$ such that the action of $\cU(\sl(n,\R))$ takes the form
\begin{eqnarray}
E_i f(\bu)&=&\sum_{k=1}^{n-i}\left(\frac{1}{2}+\bi \a_k'-\bi \sum_{j=k}^{n-i}u_{j,i+j}\right)\cdot\nonumber\\
&&\tab f(u_{k-1,i-1}+\bi ,u_{k-1,i}-\bi ,u_{k,i}-\bi ,u_{k,i+1}+\bi ),\\
F_i f(\bu)&=&\sum_{k=1}^{i}\left(\frac{1}{2}+\bi \a_k'-\bi \sum_{j=k}^{i}u_{ji}+i\sum_{j=i+1}^{n}u_{ij}\right)\cdot\nonumber\\
&&\tab f(u_{k-1,i-1}+\bi ,u_{k-1,i}-\bi ,u_{k,i}-\bi ,u_{k,i+1}+\bi ),\\
H_i f(\bu)&=&-\bi \left(\sum_{j=1}^{i-1}u_{ji}-\sum_{j=i+1}^{n}u_{ij}-\sum_{j=1}^{n-i}u_{j,i+j}\right)f(\bu).
\end{eqnarray}
Here the new $\l_k'$  is related to the old $\l_k$ by
\Eq{\l_k':=\frac{\sum_{m=1}^{n-k} m\l_{n-m}}{\sum_{m=1}^{n-k} m},}
and hence in particular $Re(\l_k')=-\frac{1}{2}$ and we set $\l'_k:=-\frac{1}{2}+\bi \a'_k$ with $\a_k'\in\R$.
Furthermore, the shifts $c_{i,j}$ obey the Haar measure on $U_{>0}^+$ (cf. Proposition \ref{Haar}), namely
\Eq{\label{Haarc}Re(c_{i,j})=\case{-\frac{1}{2}& \mbox{if $j=n$,}\\0& \mbox{otherwise.}}}
\end{Thm}
\begin{proof}
 This is an exercise in linear algebra where we require the constant in $H_i$ disappear, and for each fixed $k$, the constant in $E_i^k$ and $E_{i'}^k$ matches for every $i,i'$. This gives $n(n-1)/2$ relations in the possible $n(n-1)/2$ constants $c_{i,j}$. An elementary reduction shows that this reduces to sets of simutaneous equations of the form
$$(k-1)x_{lk}+\sum_{m=1}^{k} 2x_{lm}=2\l'_{l-k+1},\tab k=1,...,l,$$
for $1\leq l\leq n-1$, where $x_{lm}=c_{l-m+1,n-m+1}$, which obviously has a unique solution and can be solved explicitly.
\end{proof}

Therefore following the quantization procedure of $\cU_q(\sl(2,\R))$, we obtain:
\begin{Thm} \label{mainthm}The following action of the generators gives the positive principal series representation for $\cU_q(\sl(n,\R))$:
\begin{eqnarray}
E_i f(\bu)&=&\sum_{k=1}^{n-i}\left[\frac{Q}{2b}+\frac{\bi }{b}\left(\a_k'-\sum_{j=k}^{n-i}u_{j,i+j}\right)\right]_q\cdot\nonumber\\
&&\tab e^{2\pi b(-p_{k-1,i+k-1}+p_{k-1,i+k}+p_{k,i+k-1}-p_{k,i+k})},\\
F_i f(\bu)&=&\sum_{k=1}^{i}\left[\frac{Q}{2b}+\frac{\bi }{b}\left(\a_k'-\sum_{j=k}^{i}u_{ji}+\sum_{j=i+1}^{n}u_{ij}\right)\right]_q\cdot\nonumber\\
&&\tab e^{2\pi b(p_{k-1,i-1}-p_{k-1,i}-p_{k,i}+p_{k,i+1})},\\
K_i f(\bu)&=&e^{\pi b\left(\sum_{j=1}^{i-1}u_{ji}-\sum_{j=i+1}^{n}u_{ij}-\sum_{j=1}^{n-i}u_{j,i+j}\right)}f(\bu),
\end{eqnarray}
where as usual $q=e^{\pi \bi b^2}, Q=b+b\inv$ and $e^{\pm2\pi bp_{ij}}$ is shift in $u_{ij}$ by $\mp\bi b$. These operators are all positive essentially self-adjoint and satisfy the transcendental relations
\begin{eqnarray}
\label{transe}(e_i)^{1/b^2} &=& \til[e_i]\;\;,\\
\label{transf}(f_i)^{1/b^2} &=& \til[f_i]\;\;,\\
\label{transK}K_i^{1/b^2}&=&\til[K_i]\;\;,
\end{eqnarray}
where as before, $e_i=2\sin(\pi b^2)E_i$ and $f_i=2\sin(\pi b^2)F_i$ and similarly for $\til[e_i],\til[f_i]$ with $b$
replaced by $b\inv$ in all the formulas.
\end{Thm}
\begin{proof} From the CR diagram, we know that only one shifts and one weight index coincide for both $E$ and $F$. Hence the constant $\frac{Q}{2b}$ gives the positivity of the operator, using the commutation relation of $p$ and $x$:
$$q^{\frac{Q}{2b}+x}e^{2\pi bp}= \bi e^{\pi b(x+2p)},$$
so that for $c,c'$ commuting with $x$,
$$[\frac{Q}{2b}+x+c]_q e^{2\pi b(p+c')}=\left(\frac{\bi }{q-q\inv}\right)(e^{\pi b(x+2p+c+2c')}+e^{\pi b(-x+2p+c+2c')})$$
is positive. Furthermore, the two factors $q^2$ commute.

Hence let us write the operators (both $E$ and $F$) in the form
\Eq{\frac{\bi }{q-q\inv}\sum_k (A_k^++A_k^-),}
where $A_k^+A_k^-=q^2A_k^-A_k^+$. Moreover, by looking at the CR diagram, we can actually see that:
$$A_k^+ A_{k'}^\pm=q^2 A_{k'}^\pm A_k^+,$$
$$A_k^- A_{k'}^\pm=q^{-2} A_{k'}^\pm A_k^-,$$
whenever $k<k'$, so that, if we rearrange the summation as
$$A_1^++A_2^++...+A_s^++A_s^-+A_{s-1}^-+...+A_1^-,$$
then each term $q^2$ commute with the terms after that. Hence using the fact that the operators
$$e^{\pi b(\sum \a_{ij} u_{ij}+\sum \beta_{ij}p_{ij})}$$
are essentially self-adjoint, by applying Lemma \ref{b2lem} and by induction, we immediately get the required conditions,
as well as the transcendental relations.
\end{proof}

We note that it is actually impossible for the operators $(E_i,F_i,K_i)$ and $(\til[E_j],\til[F_j],\til[K_j])$ to commute in general, for example, due to relations such as
$$K_iE_{i+1}=q\inv E_{i+1}K_i,$$
since
\begin{eqnarray*}
K_i\til[E]_{i+1}&=&K_iE_{i+1}^{\frac{1}{b^2}}\\
&=&q^{-\frac{1}{b^2}}E_{i+1}^{\frac{1}{b^2}}K_i\\
&=&-\til[E]_{i+1}K_i.
\end{eqnarray*}
However we do have the following:
\begin{Prop}\label{sign} The operators $(E_i,F_i,K_i) $ commute with the generators $(\til[E_j],\til[F_j],\til[K_j])$ up to a sign.
\end{Prop}
\begin{proof} The operators commute whenever by imposing the CR diagrams, there are even numbers of $(x,p)$ pair, so that the commuting factor is of the form $q^{2\pi \bi  k}=1$ for $k\in\Z$. Otherwise for odd numbers of $(x,p)$ pair we pick up $e^{\pi \bi (2k+1)}=-1$.
Looking at the CR diagrams, it is then clear that we have:
\begin{eqnarray*}
E_i\til[E_j]&=&-\til[E_j]E_i\tab\mbox{if $|i-j|=1$,}\\
F_i\til[F_j]&=&-\til[F_j]F_i\tab\mbox{if $|i-j|=1$,}\\
E_i\til[K_j]&=&-\til[K_j]E_i\tab\mbox{if $|i-j|=1$,}\\
K_i\til[E_j]&=&-\til[E_j]K_i\tab\mbox{if $|i-j|=1$,}\\
F_i\til[K_j]&=&-\til[K_j]F_i\tab\mbox{if $|i-j|=1$,}\\
K_i\til[F_j]&=&-\til[F_j]K_i\tab\mbox{if $|i-j|=1$,}
\end{eqnarray*}
and the variables commute otherwise.
\end{proof}
In order to get commutativity, we have to modify the quantum group which will be considered in the next section. We conclude this section with several fundamental properties of these representations.

\begin{Prop} The positive representations are irreducible.
\end{Prop}
\begin{proof} First we note that the representation defined in Section \ref{sec:quan} has a formal classical limit $b\to 0$, and we know that for any fixed generic parameter $\l$ the classical action is irreducible. Since the positive representation is obtained by first rescaling the function space by $b$ and shifting of the parameter, followed by specifying the real part of $1+\l_i$ to be $\frac{1}{2}+\frac{1}{2b^2}$ which is generic since $b^2$ is irrational, it follows that the quantum representation is also irreducible.
\end{proof}

\begin{Rem} In the classical case, there is a family of intertwiners corresponding to the Weyl group elements $w\in W$ between representations of principal series parametrized by $\fh_\R^*$, see for example \cite{Kn} and references therein. It would be interesting to write explicit formula for the intertwining operators in the Mellin transform and find their $q$-deformations. In the case of $\cU(\sl(2,\R))$, the intertwining operator corresponding to the nontrivial Weyl element becomes multiplication by ratios of gamma functions, and the $q$-deformed intertwiner for $\cU_q(\sl(2,\R))$ is given by ratios of quantum dilogarithm functions \cite{PT1}. In the general case we will then obtain parametrization of the inequivalent positive representations by the positive cone $P_\R^+\simeq \fh_\R^*/W$. Thus we can restrict the values of the parameters to $\a_k'\geq 0$.
\end{Rem}

Finally we observe that the coproducts also satisfy the criterion of a positive representation:
\begin{Prop} \label{transcoprod} The coproducts
\begin{eqnarray}
\D(E_i)&=&E_i\ox K_i+1\ox E_i,\\
\D(F_i)&=&F_i\ox 1+K_i\inv\ox F_i,\\
\D(K_i)&=&K_i\ox K_i,
\end{eqnarray}
are positive essentially self-adjoint operators, and satisfy the transcendental relation
\begin{eqnarray}
(\D e_i)^{\frac{1}{b^2}}&=&\D\til[e_i],\\
(\D f_i)^{\frac{1}{b^2}}&=&\D\til[f_i],\\
(\D K_i)^{\frac{1}{b^2}}&=&\D\til[K_i].
\end{eqnarray}
\end{Prop}
\begin{proof} It follows from the fact that the two summands of the coproducts for $E_i$ and $F_i$ are positive self-adjoint and $q^2$-commute, hence we can apply Lemma \ref{b2lem} and the transcendental relations \eqref{transe}-\eqref{transK} of the generators.
\end{proof}


\section{Main theorem}\label{sec:main}
\subsection{Modified quantum group $\bU_{\fq\til[\fq]}(\sl(n,\R))$ and its positive representations}\label{sec:modified}
In order to obtain a representation of the modular double, we would like to have generators corresponding to the two parts of the modular double commute with each other. We therefore introduce the following modified quantum generators:

\begin{Def} We define $\fq:=q^2=e^{2\pi \bi b^2}$, and
\begin{eqnarray*}
\bE_i&:=&q^i E_iK_i^i,\\
\bF_i&:=&q^{i-1}F_iK_i^{1-i},\\
\bK_i&:=&K_i^2.
\end{eqnarray*}
Then the variables are positive self-adjoint. Let $[A,B]_\fq=AB-\fq\inv BA$ be the quantum commutator. Then the quantum relations in the new variables become:
\begin{eqnarray}
\label{bKE}\bK_i\bE_j&=&\fq^{a_{ij}}\bE_j\bK_i,\\
\label{bKF}\bK_i\bF_j&=&\fq^{-a_{ij}}\bF_j\bK_i,\\
\bE_i\bF_j&=&\bF_j\bE_i\tab \mbox{ if $i\neq j$},\\
{[\bE_i,\bF_i]}_\fq&=&\frac{1-\bK_i}{1-\fq},
\end{eqnarray}
and we have the quantum Serre relations
\Eq{[\bE_i,[\bE_{i+1},\bE_i]_\fq]=0=[\bE_{i+1},[\bE_{i+1},\bE_i]_\fq],}
\Eq{[\bF_i,[\bF_{i},\bF_{i+1}]_\fq]=0=[\bF_{i+1},[\bF_i,\bF_{i+1}]_\fq].}
\end{Def}

We can now state our main theorem:
\begin{Thm}
Let $\til[\fq]:=\til[q]^2=e^{2\pi \bi b^{-2}}=\fq^{\frac{1}{b^2}}$. We define the tilde part of the modified modular double by representing the generators $\til[\bE_i],\til[\bF_i],\til[\bK_i]$ using the formulas above with all the terms replaced by tilde. Then all the relations with tilde replaced hold.

Furthermore the properties of positive representations are satisfied:

\begin{itemize}
\item[(i)] the operators $\be_i,\bf_i,\bK_i$ and their tilded counterparts are represented by positive essentially self-adjoint operators,
\item[(ii)] we have the transcendental relations:
\begin{eqnarray}
(\be_i)^{\frac{1}{b^2}}&=&\til[\be_i],\\
(\bf_i)^{\;\frac{1}{b^2}}&=&\til[\bf_i],\\
(\bK_i)^{\frac{1}{b^2}}&=&\til[\bK_i].
\end{eqnarray}
\end{itemize}
Besides, all the variables $\bE_i,\bF_i,\bK_i$ commutes with all $\til[\bE_j],\til[\bF_j],\til[\bK_j]$.
Here as before, $\be_i=2\sin(\pi b^2)\bE_i$ and $\bf_i=2\sin(\pi b^2)\bF_i$ and similarly for $\til[\be_i],\til[\bf_i]$
with $b$ replaced by $b\inv$ in all the formulas. Therefore we obtain the positive representations for
$\bU_{\fq\til[\fq]}(\sl(n,\R))$.
\end{Thm}

\begin{proof} The new quantum relations follow simply by substitution and commuting $K_i$ to the same side and cancel out. For the transcendental relations, we observe that
$\bE_i^\frac{1}{b^2}=c\til[\bE]_i$ where $c$ is a constant of the form $q^n$ where $n$ is real. Since both operators are positive self-adjoint, $c=1$.
Similarly analysis hold for $\bF_i$. The case for $\bK_i$ is trivial.

For the construction, we assume
\begin{eqnarray*}
\bE_{i}&=&q^{c_i}E_iK_i^{c_i},\\
\bF_{i}&=&q^{-d_i}F_iK_i^{d_i},\\
\bK_{i}&=&K_i^2.
\end{eqnarray*}
Then the relations with $\bK_i$ is manifest. The relations involving $[\bE_i,\bF_i]_\fq$ requires $d_i=1-c_i$, and
the quantum Serre relations amount to, after reduction, the condition:
$$c_{i+1}=c_i+1,$$
while the commutativity between the original and tilde variables are governed by
$$c_{i+1}-c_i=odd,$$
which is already satisfied. Hence our choice $c_i=i$ is one of the simplest solutions to these conditions.
\end{proof}

The above modified quantum relations can be generalized to an arbitrary simply-laced type.
\begin{Prop} Given an orientation of the Dynkin diagram, we assign a weight $s_i \in\Z$ to each node $i$ such that $s_j-s_i=1$ whenever the orientation is given by
\Eq{\label{orientation}\cdots \circ^i\to\circ^j\cdots}
Then we can define the quantum generators as
\begin{eqnarray*}
\bE_i&:=&q^{s_i} E_iK_i^{s_i},\\
\bF_i&:=&q^{s_i-1}F_iK_i^{1-s_i},\\
\bK_\l&:=&K_\l^2.
\end{eqnarray*}
All relations will be the same as above, and the quantum Serre relations become
\Eq{[\bE_i,[\bE_{j},\bE_i]_\fq]=0=[\bE_{j},[\bE_{j},\bE_i]_\fq],}
\Eq{[\bF_i,[\bF_{i},\bF_{j}]_\fq]=0=[\bF_{j},[\bF_i,\bF_{j}]_\fq],}
whenever the orientation is given by \eqref{orientation}. Our choice for type $A_r$ above corresponds to the orientation given by \eqref{Arorientation}.
\end{Prop}

\begin{Rem} For an arbitrary simply-laced type, we expect that the construction of positive principal series representations can be done in an analogous way. First we express the classical positive unipotent group in terms of its cluster variable coordinate, proposed in \cite{BFZ2}. Then we look for the classical action of $\cU(\g_\R)$ in this space, as well as its Mellin transformed action. Next we proceed as before by quantizing the scalar weights, and argue using a similar diagrammatic technique that all the quantum relations are satisfied. Finally we can try to adjust the parameters so that the resulting operators are positive essentially self-adjoint, and that they are expressed as sums of $q^2$ commuting terms, which will then imply the desired transcendental relations. In general the formulas may be complicated, and we may need to use cluster variables associated to different orientations of a quiver as well as relations between them. 
\end{Rem}


\subsection{Tori realizations and the Langlands dual}\label{sec:slqnpositive}
We note that the representation for the modified quantum generators is still given by positive essentially self-adjoint operators, however they are constructed from the "half" tori $\{e^{\pi bu_{ij}}, e^{2\pi bp_{ij}}\}$. It turns out that there exists a unitary transformation that realize the action in terms of the full tori $\{e^{2\pi bu_{ij}}, e^{2\pi bp_{ij}}\}$. Let $\T_{\fq\til[\fq]}^{n(n-1)/2}$ be the quantum tori generated by positive self-adjoint $\bu_{ij},\bv_{ij},\til[\bu_{ij}], \til[\bv_{ij}]$ for $1\leq i<j\leq n$ such that
\Eq{\bu_{ij}\bv_{ij}=\fq \bv_{ij}\bu_{ij}, \tab \til[\bu_{ij}]\til[\bv_{ij}]=\til[\fq]\til[\bv_{ij}]\til[\bu_{ij}].}
which can be realized by 
\Eq{\label{fulltorus}\bu_{ij}=e^{2\pi bu_{ij}},\tab \bv_{ij}=e^{2\pi bp_{ij}},}
and similarly for $\til[\bu_{ij}],\til[\bv_{ij}]$ with $b$ replaced by $b\inv$. Then we obtain
\begin{Thm} We have an embedding
\Eq{\bU_{\fq\til[\fq]}(\sl(n,\R))\inj \T_{\fq\til[\fq]}^{n(n-1)/2}.}
\end{Thm}
\begin{proof}
There exists a transformation by multiplication of a unitary function of the form \Eq{\exp(\frac{\pi i}{2}f(\bu)),}
where $f(\bu)$ is quadratic symmetric in $u_{ij}$, which sends
$$2p_{ij}\mapsto 2p_{ij}+\sum c_{kl}u_{kl}+(j-i)\sum_{k=1}^{j}\a_k,\tab \mbox{for some }c_{kl}\in \N_{\geq 0},$$
so that the representation of $\bE_i,\bF_i$ are represented by the full torus \eqref{fulltorus}.
The explicit formula is given by
\begin{eqnarray*}
(2\sin(\pi b^2))\bE_i&=& \sum_{k=1}^{n-i}(1+qe^{-2\pi b E_i^k(\bu)}) e^{2\pi b(U_{i}^k(\bu)+E_i^k(\bp))},\\
(2\sin(\pi b^2))\bF_i&=& \sum_{k=1}^{i}(1+q\inv e^{2\pi bF_i^k(\bu)})e^{2\pi b(-U_{i-k+1}^k(\bu)+F_i^k(\bp))},
\end{eqnarray*}
where $E_i^k(\bu), F_i^k(\bu)$ are the weights, and $E_i^k(\bp),F_i^k(\bp)$ are the shifts, which are the same as before (cf. Theorem \ref{mainthm}), while
$$U_{i}^k(\bu)=\left(\sum_{m=1}^{k-1}\sum_{n=i}^{m+i} -\sum_{m=i}^{k-1}\sum_{n=i}^{i+k-1}\right)u_{mn}.$$

Then all properties of positive representations are preserved.
\end{proof}

Finally, we calculate the commutant of this representation.
\begin{Prop}\label{Langlands}The commutant for the positive representation of $\bU_{q}(\sl(n,\R)$ is generated by $\til[\bE_i],\til[\bF_i]$ and elements of the form
\Eq{\label{KKKK}\til[\bK]_1^\frac{k}{n}\til[\bK]_2^{\frac{2k}{n}}\cdots\til[\bK]_{n-k}^\frac{(n-k)k}{n}\cdots\til[\bK]_{n-2}^\frac{2(n-k)}{n}\til[\bK]_{n-1}^\frac{n-k}{n},}
for $1\leq k\leq n-1$ an integer.
\end{Prop}
\begin{proof} Since $\til[\bE_i],\til[\bF_i]$ does not commute with $\bE_j,\bF_j$ in the strong sense, any fractional powers of $\til[\bE_i]$ and $\til[\bF_i]$ will not commute simutaneously with each individual components $\bE_i^k$ and $\bF_i^k$. For the $\til[\bK_i]$ generators, using \eqref{bKE}, \eqref{bKF} amounts to solving a standard $n-1$ simutaneous set of linear equations in $n-1$ variables.
\end{proof}

\begin{Rem}Note that these elements \eqref{KKKK} correspond to the fundamental weights while $\bK_i$ correspond to simple roots, thus $\bU_\fq(\g_\R)$ can be viewed as the "adjoint" quantum group. One can also define its Langlands dual the "simply-connected" quantum group $\bU_\fq({}^L\g_\R)$ by adjoining elements of the form given in \eqref{KKKK}. Then Proposition \ref{Langlands} can be interpreted as the statement that the commutant of $\bU_{\fq}(\g_\R)$ is in fact its Langlands dual quantum group  $\bU_{\til[\fq]}({}^L\g_\R)$.
\end{Rem}


\section{Future perspectives\label{sec:future}}
Our construction of the positive principal series representations for the modular double $\bU_{\fq\til[\fq]}(\g_\R)$ as a certain $q$-deformation of the minimal principal series for $\g_\R$ suggests a strong parallel between the quantum and classical theories, similar to the parallel between finite dimensional representations of $\cU_q(\g_c)$ and $\g_c$. However in the split real case there is also a fundamental difference between the quantum and classical theories first observed by Ponsot and Teschner \cite{PT2} for $\sl(2,\R)$: the positive principal series representation of the modular double are closed under the tensor product in the sense of the direct integral decomposition. We conjecture that the closure of the positive principal series representations of the modular double is still valid for the higher rank case. One way to prove this conjecture would be to show first that the additional properties (i) and (ii) in the introduction satisfied for the positive representations characterize this class. In fact it is easy to show that both properties are preserved under the tensor product, namely from Proposition \ref{transcoprod} one has
\Eq{(\D\be_j)^{\frac{1}{b^2}}=\D\til[\be]_j,\tab(\D \bf_j)^{\frac{1}{b^2}}=\D\til[\bf]_j,}
which immediately imply the conjecture about the closure of the tensor product. Another approach to the proof would be to use a realization of positive representations in the quantum counterparts of the regular $L^2(G_\R)$ or quasi regular $L^2(G_\R/ K)$.

While the tensor product structure of the positive principal series representations destroys the parallel between the quantum and classical theories for the split real algebras, it creates a remarkable parallel between representation theories of the quantum group $\cU_q(\g_c)$ and the modular double $\bU_{\fq\til[\fq]}(\g_\R)$. In fact thanks to the closure of tensor products of positive representations one can define a continuous version of the braided tensor category for $\bU_{\fq\til[\fq]}(\g_\R)$ following the well established example of $\cU_q(\g_c)$. In the case of $\g_\R=\sl(2,\R)$ the structure of the braided tensor category has been extensively studied by Teschner et al \cite{BT, PT1,PT2}. A generalization of their results to an arbitrary simply-laced case is an interesting direction for the future research.

The existence of the braided tensor category of positive representations of the modular double $\bU_{\fq\til[\fq]}(\g_\R)$ opens an extensive program proposed by the first author in \cite{F}. Namely one can try to find the analogues of various remarkable results and constructions that were discovered and studied in relation to the braided tensor category of the finite-dimensional representations of the quantum group $\cU_q(\g_c)$. This program is not entirely new since the different partial results already exist primarily in the case of $\g_\R=\sl(2,\R)$, and it was behind the work of Teschner et al. However our construction of the positive representations for higher rank algebras strongly indicates that all the results for $\cU_{q\til[q]}(\sl(2,\R))$ can be generalized to other types of split real quantum groups and therefore one can envision future perspectives for the positive principal series representations comparable to the past developments related to finite-dimensional representations of the quantum groups initiated by Drinfeld and Jimbo.

 In particular, we would like to mention the following three  directions:
\begin{itemize}
\item[(1)] Topological quantum field theory and Chern-Simons-Witten theory,
\item[(2)] Equivalence of categories of affine Lie algebras and quantum groups,
\item[(3)] Geometrization and categorification of quantum groups and its representations.
\end{itemize}

Below we make a few comments on the versions of the above three directions for the modular double $\bU_{\fq\til[\fq]}(\g_\R)$ and the positive principal series representations.

(1) The first examples of topological quantum field theory (TQFT) introduced in \cite{At,Se,Wi1} were based on a subquotient category of finite-dimensional representations of the quantum group $\cU_q(\g)$ at the root of unity $q$ \cite{RT}. An alternative geometric approach to the same class of TQFT's has been suggested by Witten \cite{Wi2} and is known as the Chern-Simons-Witten (CSW) model for a compact group $G_c$. In the split real case the category of positive representations of $\cU_{q\til[q]}(\sl(2,\R))$ studied in \cite{BT,PT1,PT2} was suggested by Teschner (see introduction in \cite{Te}) as an alternative approach to the construction of a new class of TQFT's that arise from the quantization of the Teichm\"{u}ller spaces \cite{CF,K}. This construction of TQFT's has been completed by R. Raj \cite{R}. The geometric approach based on CSW model for a split real group $G_\R$ has been extensively studied in a recent work \cite{Di}. It is intimately related to three dimensional hyperbolic geometry and is still in the beginning of its development.

(2) The equivalence of categories of highest weight representations of affine Lie algebras and quantum groups were extensively studied in \cite{KL}. The explicit construction of the equivalence can be simplified by considering an additional category of representations of $W$-algebras, see \cite{St}. In the split real case it is still an open problem to construct a principal series of representations of the affine Lie algebra $\hat{\g}_\R$ even for the case $\g_\R=\sl(2,\R)$. However one can discuss the equivalence of categories of representations of the modular double $\bU_{\fq\til[\fq]}(\g_\R)$ and the $W$-algebra associated to $\g_\R$. In the case when $\g_\R=\sl(2,\R)$ the $W$-algebra is the Virasoro algebra and there is a strong evidence that the appropriate category of representations of the Virasoro algebra is associated to the Liouville model \cite{Te}.

(3) The first geometric construction of the finite-dimensional representations of $\cU_q(\g)$ (as well as their affine counterparts) based on the gauge theory has been discovered by Nakajima in \cite{Na}. By considering various categories of sheaves on the Nakajima varieties one obtains a categorification of these representations \cite{CKL}. Since the work of Nakajima other geometric and categorical constructions of finite-dimensional representations of $\cU_q(\g)$ have been found \cite{FFFR, Li}. In the past year physicists have observed a remarkable relation between CSW theory for the split real group $G_\R$ and the $N=2$ super-symmetric gauge theory on a three-dimensional sphere \cite{DGG,TY}. This work can be considered as a first step towards a geometrization of the category of positive representations of the modular double \'{a} la Nakajima, but a lot more work is needed to get full analogues of geometrization and categorification of the finite-dimensional representations.

The three directions of development of positive representations of the modular double of a quantum group can be complemented by various others, such as the study of these representations at roots of unity, generalizations to the affine and Kac-Moody types, counterparts of the geometric realizations of quantum groups via homology of configuration spaces, and any directions that the reader can suggest in addition. Although one cannot predict which of these directions will be particularly fruitful, it is clear that we are entering a new stage in the representation theory of quantum groups.

\end{document}